\documentclass[12pt, a4wide]{amsart}
\usepackage{ifthen,amssymb,graphicx}
\usepackage[all]{xy}
\usepackage[usenames]{color}
\usepackage{mathrsfs}

\usepackage[
    inline,
    final
    ]{ showlabels }

\newcommand{\baselocus}{\operatorname{\mathsf{Bs}}}
\newcommand{\parabaselocus}{\operatorname{\mathsf{PBs}}}

 \newcommand{\bfk}{\ensuremath{\mathbf{k}}} 
  
\newcommand{\catD}{{D}}
\newcommand{\cdbc}[1]{{D}^b_c(#1)}

\newcommand{\Hom}{{\operatorname{Hom}}}

\newcommand{\Perf}{{\operatorname{Perf }}}

\newcommand{\OO}{\mathcal{O}}

\DeclareMathOperator{\Pic}{{Pic}}

\DeclareMathOperator{\Spec}{{Spec}}

\DeclareMathOperator{\Kos}{{Kos}}

\DeclareMathOperator{\supp}{{supp}}

\newcommand{\bbR}{{\mathbb R}}

\newcommand{\cA}{{\mathcal A}}
\newcommand{\cB}{{\mathcal B}}

\newcommand{\cO}{{\mathcal O}}

\newcommand{\bL}{{\mathbf L}}

\newcommand{\pun}{{\scriptscriptstyle\bullet}}
\newcommand{\cplx}[1]{{{\mathcal #1}^{\scriptscriptstyle\bullet}}}

\newtheorem{thm}{Theorem}[section]
\newtheorem*{thm*}{Theorem}
\newtheorem{cor}[thm]{Corollary}
\newtheorem{lem}[thm]{Lemma}
\newtheorem{prop}[thm]{Proposition}

\theoremstyle{definition}
\newtheorem{defin}[thm]{Definition}
\theoremstyle{remark}
\newtheorem{rema}[thm]{Remark}
\newtheorem{eg}[thm]{Example}
\newenvironment{rem}{\begin{rema}}{\hfill\hspace{1pt}$\triangle$\end{rema}}

\numberwithin{equation}{section} 

\begin{document}
\title[Indecomposability of derived categories]{Indecomposability of derived categories for arbitrary schemes}

\author[A. C. L\'opez Mart\'{\i}n, F. Sancho de Salas]{Ana Cristina L\'opez Mart\'{\i}n and Fernando Sancho de Salas}

\email{anacris@usal.es, fsancho@usal.es}

\thanks {{\it Author's address: }Departamento de Matem\'aticas, Universidad de Salamanca, Plaza de la Merced 1-4, 37008, Salamanca, tel: +34 923294456; fax +34 923294583.\\
Work supported by Grant PID2021-128665NB-I00 funded by MCIN/AEI/ 10.13039/501100011033 and, as appropriate, by “ERDF A way of making Europe”.}
\date{\today}

\begin{abstract}  We extend the criterion of Kawatani and Okawa for indecomposability of the derived category of a smooth projective variety to arbitrary schemes. For relative  schemes, we also give a criterion for the nonexistence of semiorthogonal decompositions that are linear over the base. These criteria are based on the base loci of the global or relative dualizing complexes.
\end{abstract}


\maketitle
\setcounter{tocdepth}{1}

{\small \tableofcontents }

\section*{Introduction} The  derived category of  sheaves of a  scheme  is a large invariant that encodes a lot of geometric information about the scheme and many other invariants like the Hochschild homology and the cohomology can be extracted out from it. However, working with this huge invariant is difficult and it is important to decompose it into smaller pieces. Naively, we can consider orthogonal decompositions of the derived category, but this notion is very restrictive because for connected schemes, any orthogonal decomposition is trivial (see Proposition \ref{p:conexo}).  The notion of semiorthogonal decomposition is more flexible and it seems to be the appropiate as semiorthogonal components are proved to be important tools in non-commutative algebraic geometry (see \cite{Kuz14} for a review).

Nevertheless, there are many examples of  projective varieties whose derived category has no non-trivial semiorthogonal decompositions and this question is directly related to the birational geometry  of the variety.

For a smooth projective variety $X$, Kawatani and Okawa gave in \cite{KawOka22} a general criterion that proves that the nonexistence of semiorthogonal decompositions of $\cdbc{X}$ is governed by the base locus of the canonical sheaf $\omega_X$ of $X$. The result allowed to give many examples of smooth varieties that are minimal in the minimal model program. The paracanonical base locus and the relative base locus of the canonical line bundle of the albanese morphism are also important in the context of semiorthogonal decompositions, as it is shown in recents works by Caucci and Lin  \cite{Lin2021, Caucci22}.

 Since schemes with singularities have to be allowed in the MMP, it is also an interesting problem to determine  for a singular scheme $X$, if its category of perfect complexes $\Perf({X})$ (resp.  its bounded derived category of coherent sheaves $\cdbc{X}$) is indecomposable (resp. weakly indecomposable), that is, it admits no non-trivial  (resp. admisible) semiorthogonal decompositions. For Cohen-Macaulay varieties over a field, this problem was considered by Spence in \cite{Spen22}.  In this paper, we tackle the problem for arbitrary schemes both in the absolute and in the relative situation.

On the one hand, we give the following generalization of Kawatani and Okawa's criterion to arbitrary (noetherian)  schemes.

\medskip
\noindent{\bf Theorem} (Theorem \ref{t:indescomPerf}). {\em  
Let $X$ be a  connected   scheme,   $k $ a ring  and $\pi\colon X\to\Spec k$ a proper morphism. Let  $\Perf(X)=\langle \mathcal{A},\mathcal{B}\rangle$ be a SOD of the subcategory of perfect complexes of $X$. Then, 

{\rm (1)} For each irreducible component $X_i$ of $X$, one of the following holds:
\begin{enumerate}\item[(1.1)] $\Kos({\bf f}_x)\in \mathcal{A}$ for any closed point $x\in X_i- \text{\rm Bs}|\pi|$ (Definition \ref{relativebaselocus}) and any system of parameters ${\bf f}_x$ of $\OO_{X,x}$. 
\item[(1.2)] $\Kos({\bf f}_x)\in \mathcal{B}$ for any closed point $x\in X_i- \text{\rm Bs}|\pi|$ and any system of parameters ${\bf f}_x$ of $\OO_{X,x}$. 
\end{enumerate}

{\rm (2)} If {\rm (1.1)} (resp. {\rm (1.2)}) is satisfied, for any $b\in\mathcal{B}$ (resp.  $a\in\mathcal{A}$), the support of $b|_{X_i}$ (resp. $a|_{X_i}$) is contained in $\text{\rm Bs}|\pi|\cap X_i$.

If we assume further that $X-\text{\rm Bs}|\pi|$ is connected, then  {\rm (1)}    (and consequently {\rm (2)}) holds for the whole $X$.}\medskip

The proof of this theorem  follows the same line than those in \cite{KawOka22} and \cite{Spen22}.
When $X$ is a smooth projective scheme over a field, there are two main ingredients in the proof of the Kawatani and Okawa's  criterion in \cite{KawOka22}. First, the use of the spanning class of $\cdbc{X}$ given by the skyscraper sheaves $\mathcal{O}_x$  to detect the support of an object in $\cdbc{X}$ and, second, the notion of base point of the dualizing sheaf that together with Serre duality allows to conclude. When $X$ is singular, the skyscraper sheaves are no longer a spanning class neither for $\cdbc{X}$ nor for $\Perf({X})$. For a  Cohen-Macaulay scheme $X$, the structure sheaves  $\mathcal{O}_{Z_x}$ of the locally complete intersection zero cycles supported on the closed points of $X$  do form a spanning class for both $\cdbc{X}$ and  $\Perf({X})$ (\cite{HLS08}) and they were used by Spence (\cite{Spen22}) to generalize  Kawatani and Okawa's criterion to projective Cohen-Macaulay varieties over a field. Spence himself considers the problem of definining the base locus when the dualizing sheaf is not necessarily a line bundle but just a coherent sheaf. 
If $X$ is no longer Cohen-Macaulay, we shall use the Koszul complexes introduced in \cite{S09} as a replacement of the locally complete intersection zero cycles. Essentially, for each closed point $x$ one takes a system of parameters $f_1,\dots ,f_n$ of the local ring $\OO$ at $x$ (that is, $\OO/(f_1,\dots ,f_n)$ is a zero-dimensional ring), then  the ordinary Koszul complex $\Kos(f_1,\dots,f_n)$ and extends it by zero out of $x$ to obtain the so called Koszul complexes in the paper. We shall prove in Lemma \ref{l:soporte}   that these complexes detect the support of any object in $\cdbc{X}$. Secondly, for an arbitrary scheme $X$, the dualizing complex $D_{X/k}$ (needed to apply Grothendieck duality in its most general form) is not in general a single sheaf (placed in some degree), so we need to define the base locus of an object $\mathcal{F}\in \cdbc{X}$ and the base locus $ \text{\rm Bs}|\pi|$ of a proper morphism $\pi\colon X\to S$ (see Definition \ref{relativebaselocus}).

When $S=\Spec k$ and $X$ is a proper Cohen-Macaulay scheme over $k$, $\text{\rm Bs}|\pi|$  coincides with the base locus $ \text{\rm Bs}|\omega_{X}|$ of the dualizing sheaf and the Koszul complex $\Kos({\bf f}_x)$ of a system of parameters of $\mathcal{O}_{X,x}$ is a resolution of the structure sheaf $\mathcal{O}_{Z_x}$ of the corresponding locally complete intersection zero cycle. Notice however that also in the Cohen-Macaulay case, our proof differs from Spence's proof because in \cite{Spen22} it is used implicitly the fact that if a morphism in the derived category is fibrewise null, then it is null; this is not true in general and our poof gives a solution.

Since we use neither the projectivity of the scheme nor that the scheme is over a base field, we are able to analise the case of arbitrary affine (noetherian) schemes, proving that an affine connected scheme does not admit semiorthogonal descompositions (Corollary \ref{affine:empty}). This corollary was also proved in \cite{EL2018} by different methods.

On the other hand, we also contemplate a similar criterion in the relative situation. If $f\colon X\to T$ is a morphism of schemes,  it is important to study semiorthogonal decompositions that are $T$-linear, that is, each component of the descomposition is stable under the action by pullbacks of elements in $\Perf({T})$ (see \cite{Kuz21}).  The relative criterion of indecomposability is the following 

\medskip
\noindent{\bf Theorem} (Theorem \ref{t:relativeindescomPerf}). {\em   
 Let $X$ be a connected scheme,  $f\colon X\to T$  a proper morphism of schemes. 
Let  $\Perf(X)=\langle \mathcal{A},\mathcal{B}\rangle$ be a $T$-linear SOD of the subcategory of perfect complexes of $X$. Then, 

{\rm (1)} For each irreducible component $X_i$ of $X$, one of the following holds:
\begin{enumerate}\item[(1.1)] $\Kos({\bf f}_x)\in \mathcal{A}$ for any closed point $x\in X_i- \text{\rm Bs}|f|$ and any system of parameters ${\bf f}_x$ of $\OO_{X,x}$. 
\item[(1.2)] $\Kos({\bf f}_x)\in \mathcal{B}$ for any closed point $x\in X_i- \text{\rm Bs}|f|$ and any system of parameters ${\bf f}_x$ of $\OO_{X,x}$. 
\end{enumerate}

{\rm (2)} If {\rm (1.1)} (resp. {\rm (1.2)}) is satisfied,  for any $b\in\mathcal{B}$ (resp. $a\in \mathcal{A}$), the support of $b|_{X_i}$ (resp. $a|_{X_i}$) is contained in $\text{\rm Bs}|f|\cap X_i$.

 If we assume further that  $X-\text{\rm Bs}|f|$ is connected, then  {\rm (1)}    (and consequently {\rm (2)}) holds for the whole $X$.

}\medskip

We finish the paper with some examples of schemes whose categories of perfect complexes  are indecomposable. Some of the most important are proper and connected schemes such that the base locus of the dualizing complex is finite (or empty) (Corollary \ref{c:dim0}),  and crepant resolutions of Gorenstein schemes (Corollary \ref{c:crepante}). 

In general, the indecomposability of $\Perf(X)$ implies the weakly indecomposabilty of $\cdbc{X}$. Furthermore, while this paper was in preparation, Kuznetsov and Shinder studied more deeply the relations between  admissible subcategories and semiorthogonal decompositions of $\cdbc{X}$ to those of $\Perf(X)$. In the recent paper \cite{KuzShi22},  they prove that if $ X$ is a projective scheme over a perfect field, then $\cdbc{X}$ is indecomposable if and only if $\Perf{X}$ is indecomposable.

Combining their result with our Theorem 3.2, one obtains new examples of minimal schemes in the MMP.

Finally, we want to mention that when  $X$ is a connected regular scheme and $Y$ is a scheme that admits a dualizing complex,   Theorem 2.6 has been proven independently by Okawa (see Theorem 3.10 in \cite{Oka2023}). Shortly after a preprint of this paper was made public, Okawa himself informed us about this fact. His definition of the base locus  differs from ours (Definition \ref{relativebaselocus}). Both definitions are compared in Remark  3.11 of \cite{Oka2023}.

 {\bf Acknowlegements} We would like to thank Evgeny Shinder for his comments and suggestions and for pointing out a mistake in the proof of Theorem \ref{t:indescomPerf} in an earlier version of the paper. We are grateful to Alexander Kuznetsov for his comments and for reading the manuscript.  We also thank Shinnousuke Okawa for his interest in this work and for sharing with us the draft of his article.

\subsubsection*{Conventions}

In this paper, all schemes are assumed to be noetherian.  For any scheme $X$ we denote by $\catD(X)$ the
derived category of complexes of $\cO_X$-modules with
quasi-coherent cohomology sheaves,   $\cdbc{X}$ the full subcategory of complexes with bounded and coherent cohomology and  $\Perf{(X)}$  the full  subcategory of perfect complexes.  

\section{Preliminares}

\subsection{Grothendieck duality}
Let $X$ be a scheme, $a,b\in\catD(X)$. We shall denote by $\mathbb{R}\cplx{Hom}(a,b)$ the  derived complex of sheaves of homomorphisms, $\mathbb{R}\cplx{\Hom}(a,b)$ the derived complex of homomorphisms, and $a\overset{\mathbb{L}}\otimes b$ the derived tensor product. One has
\[ \aligned \mathbb{R}\Gamma(X,  \mathbb{R}\cplx{Hom}(a,b)) & \simeq \mathbb{R}\cplx{\Hom}(a,b)
\\ \mathbb{R}\cplx{\Hom}(a,\mathbb{R}\cplx{Hom}(b,c) ) & \simeq \mathbb{R}\cplx{\Hom}(a\overset{\mathbb{L}}\otimes b,c).
\endaligned\]
For any $a\in\catD(X)$, $a^\vee$ denotes its derived dual: $a^\vee :=\mathbb{R}\cplx{Hom}(a,\mathcal{O}_X)$. If $a$ is perfect, so is $a^\vee$ and $a\overset\sim\to a^{\vee\vee}$. The following formulas will be used 
\[\aligned  &\mathbb{R}\cplx{\Hom}(a\overset{\mathbb{L}}\otimes c,b)  = \mathbb{R}\cplx{\Hom}(a,b\overset{\mathbb{L}}\otimes c^\vee),\qquad a,b\in\catD(X), c\in\Perf{(X)}
\\  &\mathbb{R}\cplx{\Hom}(a,b) =\mathbb{R}\cplx{\Hom}(b^\vee,a^\vee),\qquad a,b\in\Perf{(X)}
\\  &\mathbb{R}\cplx{\Hom}(a,b)=\mathbb{R}\Gamma(X,b\overset{\mathbb{L}}\otimes a^\vee),\qquad a \in\Perf{(X)}, b\in\catD(X).
\endaligned\]

\subsubsection{Grothendieck Duality} For any morphism of schemes $f\colon X\to T$ the direct image $\mathbb{R}f_*\colon \catD(X)\to\catD(T)$ has a right adjoint (see for example \cite{Lip})
\[ f^!\colon \catD(T)\to\catD(X)\] and then \[\mathbb{R}\cplx{\Hom}(\mathbb{R}f_*(a),t) = \mathbb{R}\cplx{\Hom}( a ,f^!t)\] for any $a\in \catD(X),t\in\catD(T)$. The complex $D_{X/T}:=f^!\mathcal{O}_T$ is called the {\em relative dualizing complex} of $X$ over $T$. If $D_{X/T}$ is just a sheaf (placed in some degree), i.e., 
$ D_{X/T}=\omega_{X/T}[d]$ for some sheaf $\omega_{X/T}$ on $X$ and some integer $d$, then $\omega_{X/T}$ is called {\em relative dualizing sheaf of $X$ over $T$}.

  If $T=\Spec k$ is an affine scheme, then $\catD(T)=D(k)$ is the derived category of complexes of $k$-modules, and $\mathbb{R} f_*=\mathbb{R}\Gamma(X,\quad)$. Then
  \[ \mathbb{R}\cplx{\Hom}(a,b\overset{\mathbb L}\otimes f^!t)= \mathbb{R}\cplx{\Hom}(\mathbb{R}\cplx{\Hom}(b,a),t) \] for any $a\in\catD(X)$, $b\in\Perf{(X)}$, $t\in\catD(k)$.

\subsection{Zero loci of sections of coherent sheaves and base loci}

Let $X$ be a noetherian scheme, let $\mathcal{F}$ be a coherent sheaf on $X$ and $s\in H^0(X, \mathcal{F})$ a section. The {\it zero locus of $s$ }is defined as  the set
$$V(s)=\{x\in X:  s(x)\colon k(x)\to \mathcal{F}\otimes k(x)\text{ is the  zero morphism} \}$$ where $k(x)$ denotes the residue field of the point $x\in X$. When $\mathcal{F}$ is a locally free sheaf, $V(s)$ is  a closed subset of $X$. However, if $\mathcal{F}$ is not a locally free sheaf, this zero locus need not be a closed  subset of $X$, but just a constructible subset of $X$ for the Zariski topology.

\begin{defin}\label{baselocus} Let $X$ be a noetherian scheme and $\mathcal{F}$ a coherent sheaf on $X$. The {\it base locus} of $\mathcal{F}$ is defined as 
$$\text{Bs} |\mathcal{F}|:=\bigcap _{s\in H^0(X,\mathcal{F})} \overline{V(s)}$$
 where $ \overline{V(s)}$ denotes the closure of $V(s)$ in $X$.\end{defin}

$\text{Bs} |\mathcal{F}|$ is then a closed subset of $X$ whose complementary open  set is given by 
$$X- \text{Bs} |\mathcal{F}|=\bigcup_{s\in H^0(X, \mathcal{F})}(\overset{\circ}{D(s)})$$
where  $D(s):=X- V(s)=\{x\in X \text{ such that }s(x)\neq 0\}$ is the complementary subset of $V(s)$ and $\overset{\circ}{D(s)}$ is its interior. 

This can be generalized to an object $\mathcal{F}\in \catD(X)$ in the following way. Let $i\in\mathbb{Z}$ and   $s\in H^i(X,\mathcal{F}):=H^i[\mathbb{R}\Gamma(X,\mathcal{F})]=\Hom_{\catD(X)}(\OO_X,\mathcal{F}[i])$; thus, we think $s$ as a morphism $\OO_X\to \mathcal{F}[i]$ in the derived category. Let $x$ be a point of $X$ and $i_x\colon \Spec k(x)\to X$  the natural morphism. We have then a morphism
\[ s(x):= \mathbb{L}i_x^*(s)\colon k(x)\to \mathbb{L} i_x^*\mathcal{F}[i]\]  and we  define
$$V(s)=\{x\in X:  s(x)  \text{ is the zero morphism}  \}.$$   In other words, $V(s)$ is the set of points $x\in X$ such that the morphism of $k(x)$-vector spaces $k(x)\to H^i( \mathbb{L} i_x^*\mathcal{F})$ is null. 
Now, the base locus of $\mathcal{F}$ is defined as: 
$$\text{Bs} |\mathcal{F}|:= \bigcap_{i\in\mathbb{Z}} \bigcap_{s\in H^i(X,\mathcal{F})} \overline{V(s)}.$$ Then $  \text{Bs} |\mathcal{F}|=\text{Bs} |\mathcal{F}[i]|$ for any $i\in\mathbb{Z}$. 
If $\mathcal{F}$ is just a coherent sheaf, this definition agrees with Definition  \ref{baselocus}. Again, $$X- \text{Bs} |\mathcal{F}|=\bigcup_{i\in\mathbb{Z}}\bigcup_{s\in H^i(X, \mathcal{F})}(\overset{\circ}{D(s)})$$
where  $D(s):=X- V(s)=\{x\in X \text{ such that }s(x)\neq 0\}$ is the complementary subset of $V(s)$ and $\overset{\circ}{D(s)}$ is its interior. 

\begin{defin}\label{relativebaselocus} Let $f\colon X\to S$ be a proper morphism and $f^!\colon \catD(S)\to \catD(X)$ the right adjoint of $\mathbb{R} f_*$. We say that a point $x\in X$ {\em is not a base point of $f$}, if there exists $\mathcal{F}\in\cdbc{S}$ such that $x$ is not a base point of $f^!\mathcal{F}$ (i.e., there exists $s\in H^i(X,f^!\mathcal{F})$ such that  $x\in \overset{\circ}{D(s)}$). On the contrary, we say that $x$ is a base point of $f$ and denote $\text{Bs} |f|$ the set of base points of $f$.
\end{defin}

\begin{eg} (1) Assume that $S=\Spec k$, where $k$ is a field.  Then   it is quite easy to see that the base locus of $f$ is the base locus of the dualizing complex $D_{X/k}$:
\[ \text{Bs} |f|=\text{Bs} |D_{X/k}|.\] Moreover, if $X$ is Cohen-Macaulay, then $D_{X/k}=\omega_{X/k}[n]$ and $\text{Bs} |f|=\text{Bs} |\omega_{X/k}|$.

(2) In definition \ref{relativebaselocus} we could take all $\mathcal F\in\catD^+(S)$ instead of $\mathcal F\in \cdbc{S}$, and then obtain a better (i.e., smaller) base locus $\text{Bs}'|f|$.   This base locus is local on $S$: if $V$ is an open subset and $f_V\colon f^{-1}(V)\to V$ is the restriction of $f$ to $f^{-1}(V)$, then
\begin{equation}\label{localbasepoint} \text{Bs}'|f|\cap f^{-1}(V) = \text{Bs}'|f_V|.\end{equation} Indeed, let us denote $i\colon V\hookrightarrow S$, $j\colon U=f^{-1}(V)\hookrightarrow X$. Then a morphism $s\colon \OO_{U}\to f_V^!\mathcal F$ is the restriction to $U$ of a morphism $s'\colon \OO_X\to f^!\mathbb{R} i_*\mathcal F$, because
\[ \aligned\Hom_{\catD (U)}(\OO_U, f_V^!\mathcal F) &= \Hom_{\catD (X)}(\OO_X,\mathbb{R}j_* f_V^!\mathcal F)\\ &=\Hom_{\catD (X)}(\OO_X, f^!\mathbb{R}i_*  \mathcal F).\endaligned\] This yields equation \eqref{localbasepoint}.

If we assume that $f\colon X\to S$ is of finite homological dimension, then  $\text{Bs}|f|= \text{Bs}'|f|$. This follows from the fact that any morphism $\OO_X\to f^!\mathcal F$, with $F\in\catD^+ (X)$, factors as a composition $\OO_X\to f^!\mathcal F'\to f^!\mathcal F$, for some morphism $\mathcal F'\to\mathcal F$ with $\mathcal F'\in\cdbc{S}$; indeed, the morphism $\mathbb{R}f_*\OO_X\to\mathcal F$ factors through some $\mathcal F'\in\cdbc{S}$ because  $\mathbb{R} f_*\OO_X$ is perfect, since  $f$ is of finite homological dimension.
\end{eg}

\subsection{Orthogonal and semiorthogonal decompositions}
\begin{defin}\label{d:SOD}Let $\mathcal{T}$ be a triangulated category. A {\it semiorthogonal decomposition} (SOD for short) of $\mathcal{T}$,  $\mathcal{T}=\langle \mathcal{A},\mathcal{B}\rangle$, is a pair of full triangulated subcategories $\mathcal{A}$ and $\mathcal{B}$ of $\mathcal{T}$ such that:
\begin{enumerate}\item $\Hom_\mathcal{T}(b, a)=0$ for all $a\in \mathcal{A}$ and $b\in \mathcal{B}$.
\item The smallest full triangulated subcategory of $\mathcal{T}$ generated by $\mathcal{A}$ and $\mathcal{B}$ is $\mathcal{T}$, that is, for every $t\in \mathcal{T}$ there is an exact triangle $$b\to t\to a\to b[1]$$ with $b\in \mathcal{B}$ and $a\in \mathcal{A}$.
\end{enumerate}
\end{defin} 

\begin{rema} The exact triangle in condition (2)  is unique up to a unique isomorphism and if we denote $i\colon \mathcal{A}\hookrightarrow \mathcal{T}$ and $j\colon \mathcal{B}\hookrightarrow \mathcal{T}$ the inclusion functors of the two subcategories, (2) is equivalent to any of the following conditions:
\begin{enumerate} \item[(i)] $\mathcal{B}=^\perp\mathcal{A}$ and the functor $i$ has a left adjoint $G\colon \mathcal{T}\to \mathcal{A}$.
 \item[(ii)] $\mathcal{A}=\mathcal{B}^\perp$ and the functor $j$ has a right adjoint $H\colon \mathcal{T}\to \mathcal{B}$.
\end{enumerate}

\end{rema}

\begin{rema} If moreover in Definition \ref{d:SOD} we require that $\Hom_\mathcal{T}(a, b)=0$ for all $a\in \mathcal{A}$ and $b\in \mathcal{B}$, the decomposition is called an {\it orthogonal decomposition}. In this case, the  triangle in (2) splits and we obtain $t\simeq a\oplus b$.

\end{rema}

Since $\cdbc{X}$ has no non-trivial orthogonal decomposition for any connected scheme $X$ (see Proposition \ref{p:conexo} for details), one gives the following 

\begin{defin}
We will say that a triangulated category $\mathcal{T}$ is {\it indecomposable} if any  SOD is trivial (that is, if  $\mathcal{T}=\langle \mathcal{A},\mathcal{B}\rangle$ is a SOD, then one of the two components $\mathcal{A}$ or $\mathcal{B}$ is equal to zero). 

A subcategory $\mathcal{A}\subset \mathcal{T}$ is said to be  {\it left admissible} (resp. right admissible) if the inclusion functor $i\colon \mathcal{A}\hookrightarrow \mathcal{T}$ admits a left adjoint (resp. right adjoint).    $\mathcal{A}\subset \mathcal{T}$ is {\it admissible} if it is both left and right admissible. 

A SOD $\mathcal{T}=\langle \mathcal{A},\mathcal{B}\rangle$ is called {\it admissible} if both subcategories $\mathcal{A}$ and $\mathcal{B}$ are admissible.

\end{defin}

Notice that if $X$ is a smooth and proper variety,  any SOD of $\cdbc{X}$ is admissible \cite{BK89}. This is not longer true if $X$ is a singular scheme. In this case, $\Perf({X})$ is an example of a non-admissible subcategory of $\cdbc{X}$. 
Remember that we say that a scheme $X$ satisfies the (ELF) condition, if it is separated, noetherian of finite Krull dimension and its category of coherent sheaves has enough locally free
sheaves. This is satisfied for instance, for any quasi-projective scheme. In \cite{Or06}, Orlov proves the  following interesting result.

\begin{lem}\label{l: SODinducidas} Let $X$ be a scheme satisfying (ELF) condition and let $\cdbc{X}=\langle \mathcal{A},\mathcal{B}\rangle$ be an admissible SOD. Then  $$\Perf(X)=\langle \mathcal{A}\cap \Perf(X),\mathcal{B}\cap \Perf(X)\rangle$$
is a SOD  of $\Perf(X)$.
\end{lem}

This means that in the geometric context the most interesting SOD's are  admissible SOD's so that the following notion was introduced in \cite{Spen22}.

\begin{defin} We will say that a triangulated category $\mathcal{T}$ is {\it weakly indecomposable} if any admissible SOD is trivial.
\end{defin}

\subsection{Detecting the support of a complex}
Let $X$ be a scheme. In this subsection, we provide a class of objects in $\Perf{(X)}$ that are a spanning class for both $\Perf{(X)}$ and $\cdbc{X}$ and that are able to detect the support of any object $a\in \cdbc{X}$.

Let us start by remembering the definition of spanning class. 
\begin{defin} If $\mathcal{T}$ is a triangulated category, a class of objects $\Omega\subset \mathcal{T}$ is a {\it spanning class} for $\mathcal{T}$ if the following conditions are satisfied:
\begin{enumerate} \item If  $ \Hom^i(a,\omega)= 0$ for every $\omega\in \Omega$, $i\in\mathbb Z$, then $a=0$.
\item If  $\Hom^i(\omega,a)= 0$ for every $\omega\in \Omega$, $i\in\mathbb Z$, then $a=0$.
\end{enumerate}
\end{defin}

\begin{defin} For any $a\in \catD(X)$, the support of $a$ is defined by $$\supp(a):=\{ x\in X: a_x\neq 0\} =\underset{i\in\mathbb Z}\bigcup \supp(H^i(a)).$$  If $a\in \cdbc{X}$, then $\supp(a)$ is a closed subset.
\end{defin}

\subsubsection{Koszul complexes} Let $\OO$   be a noetherian local ring of dimension $n$ and maximal ideal $\frak m$. Let $x$ be the closed
point.

\begin{defin} A sequence ${\bf f} = \{ f_1, \dots , f_n\}$ of $n$ elements in $\frak m$ is called a system of
parameters of $\OO$, if $\OO/(f_1, \dots , f_n)$ is a zero-dimensional ring. In other words, $(f_1, \dots , f_n)$ is an $\frak m$-primary ideal. We shall also denote $\OO/{\bf f} = \OO/(f_1, \dots , f_n)$.
\end{defin}

It is a basic fact of dimension theory that there always exists a system of parameters. In
fact, for any $\frak m$-primary ideal $I$, there exist $f_1, \dots , f_n$ in $I$ which are a system of parameters
of $\OO$.

\begin{defin} We shall denote by $\Kos ({\bf f} )$ the Koszul complex associated to a system of parameters ${\bf f}$. That is, if we define $L = \OO^{\oplus n}$ and $\omega \colon L\to\OO$ the morphism given by $f_1, \dots , f_n$, then the Koszul complex is $\bigwedge^i_{\OO}L$   in degree $-i$ and the differential $\bigwedge^i_{\OO}L\to \bigwedge^{i-1}_{\OO}L$   is the inner contraction with $\omega$.

 It is immediate to see that $\Hom^\pun(\Kos ({\bf f} ),\OO) \simeq \Kos ({\bf f} )[-n]$. The cohomology modules $H^i(\Kos ({\bf f} ))$ are supported at $x$ (indeed, they are annihilated by $(f_1, \dots , f_n)$). Moreover, $H^0(\Kos ({\bf f} )) = \OO/{\bf f}$.  

Now let $X$ be a scheme and $x\in X$ a closed point.  Let ${\bf f}_x$ be a system of parameters of $\OO_{X,x}$.  
Let $U=\Spec A$ be an affine open neightbourhood of $x$, and let us consider the natural morphisms $$\Spec(\OO_{X,x})\overset\phi\to  U\overset j\hookrightarrow X.$$  We   shall still denote by $\Kos ({\bf f}_x)$ the complex $j_!\phi_*\Kos ({\bf f}_x)$, where $j_!$ is the extension by zero (out of $U$) functor. Notice that even if in general $j_!$ doesn't preserve quasicoherence, since  $\Kos ({\bf f}_x)$  is supported at a point, in this case, $\Kos ({\bf f}_x)$ is quasicoherent and in fact it belongs to $\Perf(X)$, because it is perfect when restricting to $U$ and zero when restricting to $X-\{x\}$.
\end{defin}

For each closed point $x\in X$, we choose a system of parameters ${\bf f}_x$ of $\OO_{X,x}$. Lemma 1.9 in \cite{S09} proves that the set 
$$\Omega=\{\Kos({\bf f}_x) \text{  for all closed points } x\in X\}$$  is a spanning class for $\cdbc{X}$ and for $\Perf{(X)}$. We shall see in Lemma \ref{l:soporte} that they also detect the support of any object $a\in\cdbc{X}$. This is very close to the technique
of ``perfect spanning classes''  introduced in Section 2.2 in \cite{Kuz06}.

\begin{lem} \label{l:soplocal} Let $X$ be a scheme and $0\neq \mathcal{F}\in \cdbc{X}$. Assume that a connected component of $\supp(\mathcal{F})$ is contained in an affine open subset of $X$. Then $\mathbb{R}\Gamma(X,\mathcal{F})\neq 0$. 
\end{lem}

\begin{proof} Let $C$ be a connected component of $\supp(\mathcal{F})$ which is contained in an affine open subset $U$. Let us put $\supp(\mathcal{F})=C\sqcup C'$, where $C'$ is the union of the remaining connected components. We may assume, shrinking $U$, that $U\cap C'=\emptyset$. Indeed, let us put $U=\Spec A$, $C=(I)_0$ and $U\cap C'=(J)_0$. Since $C\cap C'=\emptyset$, one has $I+J=A$, so there exist $i\in I,j\in J$ such that $i+j=1$. Then $C$ is contained in the affine open subset $U_j:=U-(j)_0$ and $U_j\cap C'=\emptyset$. 

Then we have $X=U\cup V$, with $V=X-C$, and $\mathcal{F}|_{U\cap V} =0$. From the exact triangle
\[ \mathbb{R}\Gamma(X,\mathcal{F})\to \mathbb{R}\Gamma(U,\mathcal{F})\oplus \mathbb{R}\Gamma(V,\mathcal{F})\to \mathbb{R}\Gamma(U\cap V,\mathcal{F}) \to \mathbb{R}\Gamma(X,\mathcal{F})[1]\]
we conclude that $\mathbb{R}\Gamma(X,\mathcal{F})\simeq \mathbb{R}\Gamma(U,\mathcal{F})\oplus \mathbb{R}\Gamma(V,\mathcal{F})$, so  it suffices to prove that $\mathbb{R}\Gamma(U,\mathcal{F})\neq 0$. But this follows because $U$ is affine and $\mathcal{F}|_U\neq 0$.
\end{proof}

 The following lemma, whose proof is implicit in \cite{S09},  shows that the objects in $\Omega$ not only form a spanning class but they are also able to detect the support of any object $a\in \cdbc{X}$.

\begin{lem} \label{l:soporte} Let $X$ be a scheme, $x\in X$ a closed point and ${\bf f}_x$ a system of parameters of $\OO_{X,x}$.  For any $a\in \cdbc{X}$, the following statements are equivalent:
\begin{enumerate} 
\item $x\in \supp(a)$.
\item $a\overset{\mathbb{L}}\otimes \Kos({\bf f}_x)\neq 0$.
\item $\mathbb{R}\cplx{\Hom}( \Kos({\bf f}_x),a)\neq 0$.
\item $\mathbb{R}\cplx{\Hom}(a,\Kos({\bf f}_x))\neq 0$.
\end{enumerate}
\end{lem}

\begin{proof} Let us denote $K= \Kos({\bf f}_x)$ and $n=\dim\OO_{X,x}$.

$(1)\Leftrightarrow (2)$. For any $a,b\in  \cdbc{X}$ one has $\supp(a\overset{\mathbb{L}}\otimes b)=\supp (a)\cap\supp(b)$. One concludes because  $\supp(K)=\{ x\}$.

$(2)\Leftrightarrow (3)$.  This follows from the isomorphisms
\[ \mathbb{R}\cplx{\Hom}( K,a)\simeq  \mathbb{R}\Gamma(X, a \overset{\mathbb{L}}\otimes K^\vee)\simeq  \mathbb{R}\Gamma(X, a \overset{\mathbb{L}}\otimes K )[-n]\] and Lemma \ref{l:soplocal}.

$(1)\Rightarrow (4)$. If $q_0$ is the maximum $q$ such that $x\in\supp(H^q(a))$, and $p_0$ is the minimum $p$ such that $H^p(K)\neq 0$, then  
\[ 0\neq \Hom(H^{q_0}(a ),H^{p_0}(K))=\Hom^{p_0-q_0} (a ,K).\]

$(4)\Rightarrow (2)$. This follows from the isomorphisms $\mathbb{R}\cplx{\Hom}(a,K)\simeq \mathbb{R}\cplx{\Hom}(a\overset{\mathbb{L}}\otimes K^\vee,\OO_X)\simeq  \mathbb{R}\cplx{\Hom}(a\overset{\mathbb{L}}\otimes K,\OO_X)[n]$.

\end{proof}

\begin{rem} If $a,b$ are two perfect complexes on $X$, then $\mathbb{R}\cplx{\Hom}(a,b)=\mathbb{R}\cplx{\Hom}(b^\vee,a^\vee)$. In particular, $$\mathbb{R}\cplx{\Hom}(a, \Kos({\bf f}_x))\simeq \mathbb{R}\cplx{\Hom}( \Kos({\bf f}_x)[-n],a^\vee) $$  and then $\supp(a)=\supp(a^\vee)$.
\end{rem}

\section{A criterion of indecomposability}
In this section, we characterize the connectivity of a  scheme by means of both categories $\cdbc{X}$ and $\Perf({X})$ and we give some criteria of indecomposability and weakly indecomposability.

 \begin{prop} \label{p:conexo} Let $X$ be a scheme. Then $X$ is connected if and only if $\cdbc{X}$ and $\Perf{(X})$ admit no non-trivial orthogonal decomposition.
\end{prop}

\begin{proof} If $X=X_1\sqcup X_2$ is not connected, then $\cdbc{X}=\cdbc{X_1}\oplus \cdbc{X_2}$ and 
$\Perf{(X)}=\Perf{(X_1)}\oplus \Perf{(X_2)}$ are non-trivial orthogonal decompositions. Conversely, suppose that  $\cdbc{X}=\mathcal{A}\oplus\mathcal{B}$ is an orthogonal decomposition. Let us first see that for any closed point $x$ and any system of parameters ${\bf f}_x$ of $\OO_{X,x}$, the complex $\Kos({\bf f}_x)$ belongs to $\mathcal{A}$ or $\mathcal{B}$. If this is not the case, then $\Kos({\bf f}_x)=a\oplus b$ with nonzero $a,b$. Then $a,b$ are perfect and $\supp(a)=\supp(b)=\{x\}$, so $\mathbb{R}\cplx{\Hom}(a,b)=\mathbb{R}\Gamma(X,a^\vee\overset{\mathbb L}\otimes b)\neq 0$ by Lemma \ref{l:soplocal}.  

Let us prove now that  the whole class $\Omega=\{\Kos({\bf f}_x)\}$ is entirely contained in $\mathcal{A}$ or in $\mathcal{B}$. To see this, assume for a contradiction that there exist closed points $x, x'\in X$ such that $\Kos({\bf f}_x)\in \mathcal{A}$ and $\Kos({\bf f}'_{x'})\in \mathcal{B}$. Due to Lemma  \ref{l:soporte},  for any $b\in \mathcal{B}$ (resp. $ a\in \mathcal{A}$), one has $x\notin\supp(b)$  (resp. $x'\notin\supp(a))$. Consider the object $\mathcal{O}_X\in \Perf{(X)}$ and the corresponding decomposition $\mathcal{O}_X=a\oplus b$ with $a\in \mathcal{A}$ and $b\in \mathcal{B}$. Then $X=\supp(\mathcal{O}_X)=\supp(a)\cup \supp (b)$ and $\supp(a)$, and $\supp(b)$ are proper closed subsets of $X$. We conclude because $\supp(a)$ and $\supp(b)$ are disjoint. Indeed, let $x$ be a closed point in the intersection;  then $\mathbb{R}\cplx{\Hom}(b,\Kos({\bf f}_x))\neq 0\neq  \mathbb{R}\cplx{\Hom}(\Kos({\bf f}_x),a)$ by  Lemma \ref{l:soporte}, and this contradicts that $\Kos({\bf f}_x)$ belongs to $\mathcal{A}$ or $\mathcal{B}$.

Hence, $\Omega$ is contained in $\mathcal{A}$ or in $\mathcal{B}$ and being a spanning class this proves that $\mathcal{B}=0$ or $\mathcal{A}=0$. 

The same proof works for $\Perf{(X)}$.

\end{proof}

\begin{lem}\label{l:neeman} Let $X$ be a scheme and $x,x'\in X$ two closed points lying in the same irreducible component $X_i$ of $X$. Let $\Perf(X)=\langle \mathcal{A},\mathcal{B}\rangle$  be a SOD of the subcategory of perfect complexes of $X$. If $\Kos({\bf f}_x)\in \mathcal{B}$ (resp. $\mathcal{A}$) for some system of parameters ${\bf f}_x$ of $\OO_{X,x}$, then $\Kos({\bf f}_{x'})\notin \mathcal{A}$ (resp. $\mathcal{B}$) for any system of parameters ${\bf f}_{x'}$ of $\OO_{X,x'}$,

\end{lem}

 \begin{proof}  

 Let $T$ be the union of the remaining irreducible components and $Y=X_i\cap T$. Then
\[ X-Y= (X_i-Y)\sqcup (X-X_i).\]  The object  $\mathcal{F}:=\mathcal{O}_{X_i-Y}\oplus\, 0$ is a perfect complex in $X-Y$. By  Theorem 2.1 (2.1.4) in \cite{Nee96}, there exists an perfect complex $\mathcal{E}\in \Perf(X)$ such that $$ \mathcal{E}|_{X-Y}\simeq \mathcal{F}\oplus \mathcal{F}[1].$$ Notice that $\supp(\mathcal{E})=X_i$.  Consider now the exact triangle associated to $\mathcal{E}\in \Perf(X)=\langle \mathcal{A},\mathcal{B}\rangle$
$$b\to \mathcal{E}\to a\to b[1]\, $$ and denote $r_i\colon X_i\hookrightarrow X$ the natural closed immersion. By applying the functor $\bL r_i^\ast \colon \Perf{X}\to \Perf{X_i}$, we get the exact triangle 
$$b|_{X_i}\to \mathcal{E}|_{X_i}\to a|_{X_i}\to b|_{X_i}[1]$$
and then $X_i=\supp (\mathcal{E}|_{X_i})=\supp(a|_{X_i})\cup \supp(b|_{X_i})$. Assume that $\Kos({\bf f}_x)\in \mathcal{B}$.  If $\Kos({\bf f}_{x'})\in \mathcal{A}$, then  the supports of $a|_{X_i}$ and $b|_{X_i}$ are proper closed subsets of $X_i$, because $x\notin \supp(a)$ and $x'\notin\supp(b)$. This contradicts  that $X_i$ is irreducible. Hence, $\Kos({\bf f}_{x'})\notin \mathcal{A}$. If $\Kos({\bf f}_x)\in \mathcal{A}$, the same argument works.
\end{proof}

We are ready to prove the main criterion of indecomposability of $\Perf(X)$ for general  schemes.

\begin{thm}\label{t:indescomPerf} Let $X$ be a  connected   scheme,   $k $ a ring  and $\pi\colon X\to\Spec k$ a proper morphism. Let  $\Perf(X)=\langle \mathcal{A},\mathcal{B}\rangle$ be a SOD of the subcategory of perfect complexes of $X$. Then, 

{\rm (1)} For each irreducible component $X_i$ of $X$, one of the following holds:
\begin{enumerate}\item[(1.1)] $\Kos({\bf f}_x)\in \mathcal{A}$ for any closed point $x\in X_i- \text{\rm Bs}|\pi|$ (Definition \ref{relativebaselocus}) and any system of parameters ${\bf f}_x$ of $\OO_{X,x}$. 
\item[(1.2)] $\Kos({\bf f}_x)\in \mathcal{B}$ for any closed point $x\in X_i- \text{\rm Bs}|\pi|$ and any system of parameters ${\bf f}_x$ of $\OO_{X,x}$. 
\end{enumerate}

{\rm (2)} If {\rm (1.1)} (resp. {\rm (1.2)}) is satisfied, for any $b\in\mathcal{B}$ (resp.  $a\in\mathcal{A}$), the support of $b|_{X_i}$ (resp. $a|_{X_i}$) is contained in $\text{\rm Bs}|\pi|\cap X_i$.

If we assume further that $X-\text{\rm Bs}|\pi|$ is connected, then  {\rm (1)}    (and consequently {\rm (2)}) holds for the whole $X$.
\end{thm}

\begin{proof} (1) Let $x\in X_i-\text{Bs}|\pi|$ be a closed point and  ${\bf f}_x$ a system of parameters of $\OO_{X,x}$.  Let us show that $K:=\Kos({\bf f}_x)$ belongs either to $\mathcal{A}$ or to $\mathcal{B}$. Due to the existence of the SOD, there is an exact triangle \begin{equation}\label{e:triangle}
b\to K\to a\xrightarrow{f}b[1]\, 
\end{equation}
with $a\in \mathcal{A}$ and $b\in \mathcal{B}$. Since $x$ is not a base point of $\pi$, there exist $N\in\cdbc{k}$ and  $s\in H^0(X,\pi^!N)$ such that $x\in U:=\overset{\mathrm{o}}{D(s)}$.   The composition $$a \xrightarrow{f } b[1]\xrightarrow{1\otimes s}b[1]\overset{\mathbb{L}}\otimes \pi^!N $$ is null, because
$$ \mathbb{R}\cplx{\Hom} (a, b[1]\overset{\mathbb{L}}\otimes \pi^!N)\simeq \mathbb{R}\cplx{\Hom} (\mathbb{R}\cplx{\Hom}(b[1],a),N)$$  
and this is equal to zero by semiorthogonality. Here our proof differs from \cite{Spen22}. It is claimed there that the vanishing of $f|_U\circ (1\otimes s|_U)$ and $s|_U\neq 0$ implies $f|_U=0$. It is true that $f(y)=0$ for any $y\in U$, but this does not imply in general that $f|_U=0$, because the vanishing of a morphism in the derived category is not detected fibrewise. We shall proceed as follows. 

\,\,\noindent{\it Claim.} Either $a|_U=0$ or $ b|_U=0$. Let us denote $V=U-\{x\}$. Then $b|_V =0$. Indeed, if $y$ is a point of $V$ in the support of $b$, then $(1\otimes s)(y)=1\otimes s(y)$ is not zero (because $s(y)\neq 0$ and we are tensoring by a nonzero object in the derived category of $k(y)$-vector spaces), but its composition with $f(y)$ is null. This is not possible because $f(y)$ is an isomorphism. We have also that $a|_V=0$, because $f$ is an isomorphism on $X-\{x\}$. 

If $a|_U\neq 0\neq  b|_U$, then $\supp(a|_U)=\supp(b|_U)=\{ x\}$ and then $\bbR\cplx{\Hom}(b,a)\neq 0$ by Lemma \ref{l:soplocal} applied to $b^\vee\overset{\mathbb{L}}\otimes a$.  Thus, the claim is proved.

If $a|_U=0$,  the morphism $K\to a$ in  \eqref{e:triangle} is zero. Hence, $b\simeq K\oplus a[-1]$. By semiorthogonality, we see that  $a[-1]=0$, because otherwise we would obtain a non-zero projection $b\to a[-1]$. Hence, $K\in \mathcal{B}$. If we instead assume $b|_U=0$, similarly we obtain $K\in \mathcal{A}$ and this proves that $K$ belongs either to $\mathcal{A}$ or to $\mathcal{B}$.

 All the objects  $\Kos({\bf f}_x)$ for  $x\in X_i-\text{Bs}|\pi|$ belongs to the same subcategory either $\mathcal{A}$ or $\mathcal{B}$ due to Lemma \ref{l:neeman}

(2) If (1.1) is true, for any object $b\in \mathcal{B}$, the semiorthogonality condition implies that $$\mathbb{R}\cplx{\Hom}   (b, \Kos({\bf f}_x))=0$$ for all  $x\in X_i- \text{Bs}|\pi|$. One concludes by  Lemma \ref{l:soporte}.

Finally, assume that  $W:=X-\text{\rm Bs}|\pi|$ is connected. From the exact triangle 
$$b\to \mathcal{O}_X\to a\to b[1]$$
associated to $\mathcal{O}_X\in \Perf(X)=\langle \mathcal{A},\mathcal{B}\rangle$
 we obtain that $X=\supp(\mathcal{O}_X)=\supp(a)\cup\supp(b)$. Notice that if $x\in \supp(a)\cap \supp(b)$, then one has that $\mathbb{R}\cplx{\Hom}(b,\Kos({\bf f}_x))\neq 0\neq \mathbb{R}\cplx{\Hom}(\Kos({\bf f}_x),a)$ by  Lemma \ref{l:soporte}, and this implies that $\Kos({\bf f}_x)$  belongs neither to $\mathcal{A}$ nor to $\mathcal{B}$. Thus $\supp(a)\cap \supp(b)\subseteq  \text{\rm Bs}|\pi|$. Restricting to $W$,  one gets that $ W=\supp(a|_W)\sqcup \supp(b|_W)$. Since $W$ is connected, one of them equals $W$. Hence, all the objects $\Kos({\bf f}_x)$ with $x\in W$ belong  to the same subcategory either $\mathcal{A}$ or $\mathcal{B}$  by Lemma \ref{l:soporte}.

\end{proof}

 As an immediate consequence we have the following 

\begin{cor} Let $X$ be an irreducible scheme,   $k $ a ring  and $\pi\colon X\to\Spec k$ a proper morphism. Let  $\Perf(X)=\langle \mathcal{A},\mathcal{B}\rangle$ be a SOD of the subcategory of perfect complexes of $X$. Then, one of the following holds:
\begin{enumerate}\item $\Kos({\bf f}_x)\in \mathcal{A}$ for any closed point $x\in X- \text{\rm Bs}|\pi|$ and any system of parameters ${\bf f}_x$ of $\OO_{X,x}$. In this case,  the support of all objects in  $\mathcal{B}$  is contained in $\text{\rm Bs}|\pi|$.
\item $\Kos({\bf f}_x)\in \mathcal{B}$ for any closed point $x\in X- \text{\rm Bs}|\pi|$ and any system of parameters ${\bf f}_x$ of $\OO_{X,x}$. In this case,  the support of all objects in $\mathcal{A}$  is contained in $\text{\rm Bs}|\pi|$.
\end{enumerate}

\end{cor}

We shall give now a relative version of this theorem. 

\begin{defin} Let $f\colon X\to T$ be a morphism of schemes. A subcategory $\mathcal{A}$ of $\Perf{(X)}$ is called {\em $T$-linear} if $a\overset{\mathbb{L}}\otimes \mathbb{L} f^*t$ belongs to $\mathcal{A}$ for any $a\in \mathcal{A}$ and any $t\in\Perf{(T)}$. A SOD  $\Perf(X)=\langle \mathcal{A},\mathcal{B}\rangle$ is called $T$-linear if $\mathcal{A},\mathcal{B}$ are $T$-linear (it suffices that one of them is).
\end{defin}

\begin{thm} \label{t:relativeindescomPerf} Let $X$ be a connected scheme,  $f\colon X\to T$  a proper morphism of schemes. 
Let  $\Perf(X)=\langle \mathcal{A},\mathcal{B}\rangle$ be a $T$-linear SOD of the subcategory of perfect complexes of $X$. Then, 

{\rm (1)} For each irreducible component $X_i$ of $X$, one of the following holds:
\begin{enumerate}\item[(1.1)] $\Kos({\bf f}_x)\in \mathcal{A}$ for any closed point $x\in X_i- \text{\rm Bs}|f|$ and any system of parameters ${\bf f}_x$ of $\OO_{X,x}$. 
\item[(1.2)] $\Kos({\bf f}_x)\in \mathcal{B}$ for any closed point $x\in X_i- \text{\rm Bs}|f|$ and any system of parameters ${\bf f}_x$ of $\OO_{X,x}$. 
\end{enumerate}

{\rm (2)} If {\rm (1.1)} (resp. {\rm (1.2)}) is satisfied,  for any $b\in\mathcal{B}$ (resp. $a\in \mathcal{A}$), the support of $b|_{X_i}$ (resp. $a|_{X_i}$) is contained in $\text{\rm Bs}|f|\cap X_i$.

 If we assume further that  $X-\text{\rm Bs}|f|$ is connected, then  {\rm (1)}    (and consequently {\rm (2)}) holds for the whole $X$.
\end{thm}

\begin{proof} The proof is the same than that of Theorem \ref{t:indescomPerf}. We only have to prove that for any $a\in\mathcal{A}, b\in\mathcal{B}, N\in\cdbc{T}$ one has
\[ \mathbb{R}\cplx{\Hom}(a,b\overset{\mathbb{L}}\otimes f^!N)=0.\]

Now, we have 
\[\aligned  \mathbb{R}\cplx{\Hom}(a,b\overset{\mathbb{L}}\otimes f^!N)&= \mathbb{R}\cplx{\Hom}(a\overset{\mathbb{L}}\otimes b^\vee, f^!N)\\ & = \mathbb{R}\cplx{\Hom}(\mathbb{R} f_*(a\overset{\mathbb{L}}\otimes b^\vee), N)\endaligned\] and it suffices to prove that $\mathbb{R} f_*(a\overset{\mathbb{L}}\otimes b^\vee)=0.$  Since $f$ is proper, one has that $\mathbb{R} f_*(a\overset{\mathbb{L}}\otimes b^\vee)\in \cdbc{T} $; by Lemma \ref{l:soporte},  it suffices to prove that $\mathbb{R} f_*(a\overset{\mathbb{L}}\otimes b^\vee)\overset{\mathbb{L}}\otimes \Kos({\bf f}_t)=0$ for any closed point $t\in T$ and any system of parameters ${\bf f}_t$ of $\OO_{T,t}$. By Lemma \ref{l:soplocal}, it suffices to prove that $$\mathbb{R}\Gamma(T, \mathbb{R} f_*(a\overset{\mathbb{L}}\otimes b^\vee)\overset{\mathbb{L}}\otimes \Kos({\bf f}_t))=0.$$  By projection formula, 
\[ \mathbb{R} f_*(a\overset{\mathbb{L}}\otimes b^\vee)\overset{\mathbb{L}}\otimes \Kos({\bf f}_t) =  \mathbb{R} f_*(a\overset{\mathbb{L}}\otimes b^\vee \overset{\mathbb{L}}\otimes \mathbb{L}f^*\Kos({\bf f}_t)) \] and then
\[\aligned  \mathbb{R}\Gamma(T, \mathbb{R} f_*(a\overset{\mathbb{L}}\otimes b^\vee)\overset{\mathbb{L}}\otimes\Kos({\bf f}_t))& =\mathbb{R}\Gamma(X,a\overset{\mathbb{L}}\otimes b^\vee \overset{\mathbb{L}}\otimes \mathbb{L}f^*\Kos({\bf f}_t))\\ & 
 = \mathbb{R}\cplx{\Hom}(b, a\overset{\mathbb{L}}\otimes \mathbb{L}f^*\Kos({\bf f}_t))\endaligned\] which is zero because $\mathcal{A}$ is $T$-linear.
\end{proof}

\begin{cor}\label{affine:empty} Let $X$ be an affine and connected  scheme. Then any SOD of $\Perf{(X)}$ is trivial.
\end{cor}

\begin{proof} Since $X$ is affine, $\pi\colon X\to \Spec k$ is an isomorphism, with $k=\Gamma(X,\mathcal{O}_X)$; moreover, the dualizing complex is trivial: $D_{X/k}=\mathcal{O}_X$, and $\mathcal{O}_X$ has no base points. One concludes by Theorem  \ref{t:indescomPerf}.
\end{proof}

This corollary is also proved in  \cite{EL2018} by different methods and it may be generalized as follows:

\begin{prop} \label{p:dim0} Let $X$ be a connected  scheme and let $Z\subset X$ be a closed subset such that each connected component of $Z$ is contained in some affine open subset of $X$.  Let $\Perf{(X)}=\langle \mathcal{A},\mathcal{B}\rangle$ be a SOD. If  the support of every object in $\mathcal{A}$ or in $\mathcal{B}$ is contained in $Z$, then the SOD is trivial.
\end{prop}
\begin{proof} If $X$ is affine, we conclude by Corollary \ref{affine:empty}. Assume now that $X$ is not affine, so $Z\neq X$.  Let us suppose that for every object $a\in \mathcal{A}$  one has $\supp(a)\subseteq Z$. Consider the object $\mathcal{O}_X\in \Perf{(X)}$ and the corresponding exact triangle $$b\to \mathcal{O}_X\to a\to b[1]\, .$$
Since $\supp(\mathcal{O}_X)=X$, then $b\neq0$. On the other hand, if $a\neq 0$, from the equality $X=\supp(a)\cup \supp(b)$ and the connectedness of $X$ we obtain that $\supp(b^\vee\overset{\mathbb{L}}\otimes a)=\supp(a)\cap\supp(b) $  is not empty. Since any connected component of $\supp(b^\vee\overset{\mathbb{L}}\otimes a)$ is contained in an affine open subset, we may apply  Lemma \ref{l:soplocal}  to $b^\vee\overset{\mathbb{L}}\otimes a$ and we obtain $\mathbb{R}\cplx{\Hom}(b,a)\neq 0$, which contradicts the semiorthogonality assumption. Hence $a=0$ and $\mathcal{O}_X\in\mathcal{B}$. Now, for any $a\in \mathcal{A}$, we have, by semiorthogonality, that $\mathbb{R}\Gamma(X,a)=\mathbb{R}\cplx{\Hom}(\mathcal{O}_X, a)=0$, and then $a=0$, again by Lemma \ref{l:soplocal}.
\end{proof}

\begin{cor} \label{c:dim0} Let $X$ be a  proper and connected scheme over a field $k$. If $\text{\rm Bs}|D_{X/k}|$ is finite (i.e., $\dim  \text{\rm Bs}|D_{X/k}|=0$) or empty, then  $\Perf(X)$ is indecomposable.   If moreover $X$ satisfies ELF condition, then $\cdbc{X}$ is weakly indecomposable.
\end{cor}

\begin{proof} Let $\pi \colon X\to \Spec k$ be the structure morphism. Since $k $ is a field, $\text{\rm Bs}|\pi|=\text{\rm Bs}|D_{X/k}|$. If it is empty, then  $\Perf(X)$ is indecomposable by Theorem  \ref{t:indescomPerf}. Assume now that $\text{\rm Bs}|D_{X/k}|$ is finite. Let $\Perf(X)=\langle \mathcal{A},\mathcal{B}\rangle$ be a SOD and consider the exact triangle $$b\to \mathcal{O}_X\to a\to b[1]$$
associated to $\mathcal{O}_X\in \Perf(X)=\langle \mathcal{A},\mathcal{B}\rangle$ so that $X=\supp(a)\cup \supp(b)$.
 Notice that if $x\in \supp(a)\cap \supp(b)$, by Lemma \ref{l:soporte} one has that $\mathbb{R}\cplx{\Hom}(b,\Kos({\bf f}_x))\neq 0\neq \mathbb{R}\cplx{\Hom}(\Kos({\bf f}_x),a)$ and this implies that $\Kos({\bf f}_x)$  belongs neither to $\mathcal{A}$ nor to $\mathcal{B}$. Then, by Theorem \ref{t:indescomPerf} we have that $\supp(a)\cap \supp(b)\subseteq  \text{\rm Bs}|D_{X/k}|$. If  $\supp(a)\cap \supp(b)\neq \emptyset$, applying
 Lemma \ref{l:soplocal} to $b^\vee\overset{\mathbb{L}}\otimes a$  we obtain that $\mathbb{R}\cplx{\Hom}(b,a)\neq 0$, which contradicts semiorthogonality.  So $\supp(a)\cap \supp(b)=\emptyset$ and the connectivity of $X$ allows to conclude that $a=0$ or $b=0$, i.e. $\mathcal{O}_X\in \mathcal{A} $ or  $\mathcal{O}_X\in \mathcal{B} $. If $\mathcal{O}_X\in \mathcal{A} $, then, for any non base point $x$, $\Kos({\bf f}_x)\in\mathcal{A}$. This implies that the support of any $b\in\mathcal{B}$ is contained in $\text{Bs}|D_{X/k}|$ and we conclude by Proposition \ref{p:dim0}. If  $\mathcal{O}_X\in \mathcal{B} $ the argument is analogous.
 
 Let us prove now that $\cdbc{X}$ is weakly indecomposable. Let $\cdbc{X}=\langle \mathcal{A},\mathcal{B}\rangle$ be an admissible SOD for $\cdbc{X}$. By  Lemma \ref{l: SODinducidas}, it induces the following  SOD for  $\Perf(X)$ 
$$\Perf(X)=\langle \mathcal{A}\cap \Perf(X),\mathcal{B}\cap \Perf(X)\rangle\, .$$ Since we already know that $\Perf{(X)}$ is indecomposable, one of the two components of the decomposition is zero. Assume that $\mathcal{B}\cap \Perf(X)=0$ and let us prove that in this case $\mathcal{B}$ is also zero. Indeed, for any $b\in \mathcal{B}$ one has $\mathbb{R}\cplx{\Hom} (b,\Kos({\bf f}_x))=0$ for any closed point $x\in X$ and any system of parameters ${\bf f}_x$ of $\OO_{X,x}$, because    $\Kos({\bf f}_x)\in \mathcal{A}$. Since these $\Kos({\bf f}_x)$  are a spanning class  for $\cdbc{X}$, this implies that $b=0$. 
\end{proof}

\begin{rem} For schemes that satisfy ELF condition, the argument in the proof of Corollary \ref{c:dim0} proves that the indecomposability of $\Perf({X})$ implies the weakly indecomposability of $\cdbc{X}$.  If moreover we assume that the scheme $X$ is projective over a perfect field,  thanks to the recent work by Kuznetsov and Shinder \cite{KuzShi22}, we also  know that the indecomposability of $\Perf({X})$ ensures the  indecomposability of  $\cdbc{X}$. 
\end{rem}
Thus, combining our results with Corollary 6.6 in \cite{KuzShi22}, we get the following

\begin{cor} \label{c:IndDC} Let $X$ be a  connected projective scheme over a perfect field. Suppose that $ \text{\rm Bs}|D_{X/k}|$ is finite or empty.  Then  $\cdbc{X}$ is indecomposable.
\end{cor}

\section{Examples}

The results in the above section allow to increase the list of schemes whose derived categories are indecomposable or weakly indecomposable.

Let us consider the case of dimension 1. If $C$ is an affine and connected curve, we know that $\Perf{(C)}$ is indecomposable by Corollary \ref{affine:empty}.
Let us assume from now on that $C$ is a projective curve over a field $k$.
If $C$ is smooth, by results of Okawa in \cite{Oka2011} we know that $\cdbc{C}=\Perf{(C)}$ is indecomposable if and only if its arithmetic genus $g_a(C)\geq 1$. For integral Cohen-Macaulay curves, Spence recently proves  the same result: If $C$ is a projective integral Cohen-Macaulay curve, $\Perf{(C)}$ is indecomposable if and only if $g_a(C)\geq 1$ (see Corollary 3.4 in \cite{Spen22}).

Corollary \ref{c:dim0} gives us the following

\begin{cor}\label{c:Kodaira} Let $C$ be a projective connected Cohen-Macaulay curve with arithmetic genus $g_a(C):=1-\chi(\mathcal{O}_C)=1$ and trivial dualizing sheaf. Then $\Perf{(C})$ is indecomposable and $\cdbc{C}$ is weakly indecomposable.
\end{cor}

Notice that this corollary applies to all Kodaira curves, that is, all possible degenerations of a smooth elliptic curve in a smooth elliptic surface. These curves were classified by Kodaira in \cite{kod} and the list includes examples of curves with nilpotents, multiple curves, and curves with singularities that are not ordinary double points.

In the case of Gorenstein curves, the base locus of the canonical sheaf was studied in \cite{Cata82, CFHR99}.
Let $C$ be a projective  Gorenstein curve over a field $k$. Let $\omega_C$ be its dualizing  sheaf. $C$ is said to be {\it numerically $m$-connected} if $$\deg( \omega_C\otimes \mathcal{O}_B)-\deg \omega_B\geq m$$ for every generically Gorenstein subcurve $B\subset C$.

\begin{rema} When $C\subset S$ is a divisor in a nonsingular surface, since $\deg( \omega_C\otimes \mathcal{O}_B)-\deg \omega_B= B\cdot (C-B)$ the above condition can be given in terms of the intersection number $B\cdot(C-B)$.
\end{rema}

\begin{cor} \label{c:2conexa} Let $C$ be a projective Gorenstein (but possibly reducible or nonreduced) curve not isomorphic to $\mathbb{P}^1$ . If $C$ is 2-connected, then $\Perf(C)$ is indecomposable  and $\cdbc{C}$ is weakly indecomposable.
\end{cor}

\begin{proof} Theorem 3.3 in \cite{CFHR99} proves that if $C$ is 2-connected, then $\text{Bs}|\omega_C|=\emptyset$ or $C\simeq \mathbb{P}^1$ so that one concludes by Corollary \ref{c:dim0}.
\end{proof}
The results by Catanese in \cite{Cata82} allow to study the case of a stable curve. 
Since any stable curve is a reduced Gorenstein projective curve, by Theorem D in  \cite{Cata82}, if $C$ is a stable curve, then the base locus $\text{Bs}|\omega_C|$ of the canonical sheaf of $C$ consists of
\begin{enumerate}\item {\it Disconnecting nodes} of $C$, that is, nodes $p\in C$ such that the normalization of $C$ at $p$ has two connected components and
\item Rational components $E$ of $C$ such that all singular points of $E$ are disconnecting nodes (following Catanese's notion, {\it loosely connected rational tails} (LCRT)).
\end{enumerate}

\begin{cor}\label{c:estable} Let $C$ be a stable curve over a field $k$ without LCRT's. 
 Then, $\Perf{(C)}$ is indecomposable and $\cdbc{C}$ is weakly indecomposable.
\end{cor}

In all this examples, we get that  $\cdbc{C}$ is indecomposable if we assume the perfectness of the base field.
Compare these results with  Corollary 6.8 in \cite{KuzShi22}.

In the relative situation, notice that if $X$ and $T$ are two Gorenstein schemes,  Theorem \ref{t:relativeindescomPerf} applies to any proper morphism $f\colon X\to T$ whose relative dualizing sheaf doesn't  have base points, for instance when it is trivial. Then one has,

\begin{cor}\label{c:crepante} If $f\colon X\to T  $  is a crepant resolution of Gorenstein schemes, then $\Perf({X})$  has no non-trivial $T$-linear SODs.
\end{cor}


\appendix

\section{Indecomposability of semistable curves\\ (by Shinnosuke Okawa)}

Let \(   X\) be a semistable curve over an algebraically closed field \(   \bfk \);
namely, a connected projective nodal curve such that the relative dualizing sheaf \(   \omega _{ X / \bfk }
\) of \(   X\) over \(   \bfk\) is nef, in the sense that \(   \deg   \omega _{ X / \bfk }   \vert    _{    C   }   \ge   0\) for any irreducible component \(   C   \hookrightarrow   X\). It is shown in Corollary \ref{c:estable} that \(   \Perf (X)\) admits no semiorthogonal decomposition provided that \(    \omega _{ X /\bfk }   \) is ample and there is no LCRT. The aim of this appendix is to answer the question by Alexander Kuznetsov whether Corollary \ref{c:estable} generalizes to all semistable curves. The author is indebted to him for asking the question and for his useful comments.
The author was partially supported by JSPS Grants-in-Aid for Scientific Research
(18H01120, 19KK0348, 20H01797, 20H01794, 21H04994, 23H01074).

\begin{prop}\label{proposition:semistable curves are indecomposable}    Let    \(      X   \)   be a nodal connected projective curve over an algebraically closed field    \(        \bfk    \).    Then    \(       \Perf       (X)    \)    is semiorthogonally indecomposable if and only if    \(       X    \)    is semistable.
\end{prop}

For the sake of completeness, we first show the following easy lemma. Below \(    \Pic    ^{     0    }    (    X   )\) denotes the identity component of the Picard scheme.

\begin{lem}\label{lemma:no global section}    Let    \(       X    \)    be a reduced and connected projective curve over    \(       \bfk    \)    and    \(       L       \in       \Pic       ^{        0       }       (        X       )    \)    such that    \(       L       \not       \simeq       \cO       _{        X       }    \).    Then    \(       H       ^{        0       }       \left(        X,        L       \right)       =       0    \).
\end{lem}

\begin{proof}    Suppose    \(       L       \in       \Pic       ^{        0       }       \left(        X       \right)    \)    admits a global section    \(       0       \neq       s       \in       H       ^{        0      }       \left(        X,        L       \right)           \).       

 If    \(       s       \vert       _{        C      }      \neq      0    \)    for any irreducible component    \(       C       \hookrightarrow       X    \),    then obviously    \(       s    \)    must be nowhere vanishing. This would imply    \(       L       \simeq       \cO       _{        X       }    \).

    Suppose that    \(       s       \vert       _{        C       }       =       0    \)    for some irreducible component      \(       C    \).    We can assume without loss of generality that there is another component    \(       D       \hookrightarrow       X    \)    which intersects    \(       C    \)    and    \(
       s       \vert       _{        D       }       \neq       0    \).    Since    \(       L       \vert       _{        D       }       \in       \Pic ^{ 0 }       \left(        D       \right)    \),    it follows that    \(        s        \vert        _{         D           } \)    is a nowhere vanishing global section of    \(        L        \vert        _{         D        }    \).    This contradicts    \(       C \cap D       \neq       \emptyset    \)    and    \(       s       \vert       _{        C       }       =       0    \).
\end{proof}

\begin{lem}\label{lemma:there is a nice line bundle}    Let    \(       X    \)    be a nodal projective curve over    \(       \bfk    \)    such that    \(       \dim       _{        \bfk       }       H ^{ 1 }       \left(        X,        \cO        _{            X        }       \right)       >       0    \).    Then there exists an invertible sheaf    \(       L       \not\simeq       \cO       _{        X       }    \)    such that    \(       L       \in       \Pic ^{ 0 } ( X )    \).
\end{lem}

\begin{proof}    Since    \(       \dim X       =       1    \),    it follows that    \(       H ^{ 2 }       \left(        X,        \cO _{ X }       \right)       =       0    \).    The deformation theory of coherent sheaves implies that    \(       \Pic ^{ 0 }       \left(        X       \right)    \)    is a smooth connected group scheme of dimension    \(        \dim        _{         \bfk        }        H ^{ 1 }        \left(         X,         \cO         _{             X         }        \right)        >        0    \).    Hence we can take     \(       L    \)    to be the line bundle corresponding to any point of    \(        \Pic ^{ 0 }        \left(         X        \right)        \setminus               \left\{           \cO _{ X }        \right\}            \).
\end{proof}

\begin{proof}[Proof of  Proposition \ref{proposition:semistable curves are indecomposable}]    Let us first assume    \(       X    \)    is semistable, and show that    \(       \Perf       (X)    \)    is indecomposable.    As in Theorem 1.4 in \cite{Lin2021}, let us consider the paracanonical base locus:
    
    \noindent    \begin{align}        \parabaselocus        (            X        )        :=        \bigcap        _{            L            \in            \Pic ^{ 0 } ( X )        }        \baselocus (            L            \otimes            \omega            _{                X                /                \bfk            }        )        \left(            \hookrightarrow            \baselocus            \omega            _{                X                /                \bfk            }            \hookrightarrow            X        \right)
    \end{align}
    As in the smooth case, if one can show that    \(        |        \parabaselocus        (            X        )        |        <        \infty    \), then the semiorthogonal indecomposability of    \(       \Perf (X)    \)    follows. This follows from that any semiorthogonal decomposition of    \(       \Perf (X)    \)    is stable under the action of    \(       \Pic       ^{        0       }       \left(        X       \right)    \)    by Theorem 3.8 in \cite{KawOka22}. Indeed, take a closed point    \(       x       \in       X       _{        \text{\rm sm}       }       \setminus       \parabaselocus ( X )    \). Let    \(       \Perf       (X)       =       \langle        \cA,        \cB        \rangle    \)    be a semiorthogonal decomposition and
    
\noindent
\begin{align}\label{equation:decomposition of k(x) bis}        b       \to        \cO       _{           x       }        \to        a        \to        b [ 1 ]    \end{align}    be the corresponding distinguished triangle. Then for any    \(       L       \in       \Pic       ^{        0       }       (        X       )    \)    it follows that    \(       b       \otimes       L       \in       \cB    \)    and    \(       a       \otimes       L       \in       \cA    \)    by the stability. Hence from~\eqref{equation:decomposition of k(x) bis}    \(       \otimes       L    \)    we get    \(       b       \otimes       L       \simeq       b    \)    and    \(       a       \otimes       L       \simeq       a    \).    By assumption there are    \(       L       \in       \Pic       ^{        0       }       (        X       )    \)    and    \(       s       \in       H ^{        0       }       \left(        X,        L        \otimes        \omega        _{            X            /            \bfk        }       \right)    \)    such that    \(       s ( x )       \neq       0    \).    Thus we get a map
    
\noindent
    \begin{align}        b        \simeq        b        \otimes        L        ^{            - 1        }        \xrightarrow[]{            s \cdot        }        b        \otimes        \omega        _{            X            /            \bfk        }
    \end{align}    which is an isomorphism on an open neighborhood of    \(       x    \).    Then by standard arguments (see, say, the proof of Theorem \ref{t:indescomPerf}) we can conclude that    \(       \cO       _{        x       }    \)    is contained in either    \(       \cA    \)    or    \(       \cB    \).    Then, finally, one can show that all points of    \(       X       _{        \text{\rm sm}       }       \setminus        \parabaselocus        (            X        )    \)    are simultaneously contained in either    \(       \cA    \)    or    \(       \cB    \)    by the arguments of Corollary \ref{c:dim0}    .
    
    In the rest we show that    \(        \parabaselocus        (            X        )    \)    is a finite set if    \(       X    \)    is a semistable curve.
    By Corollary 1.2 in \cite{Ran14}\footnote{Though Catanese states it only for stable curves in his original paper \cite{Cata82}, Ziv Ran asserts in \cite{Ran14} that the same holds for semistable curves as well. He also gives a complete proof, which applies to all semistable curves.}, it is enough to show that a general closed point of any LCRT component (or a separating line, after \cite{Ran14}) is not contained in
    \(        \parabaselocus        (            X        )    \).

    Let    \(        {\mathbb P} ^{ 1 }        \simeq       E       \hookrightarrow       X    \)    be an LCRT component. By definition, there are mutually disjoint subcurves    \(       X _{ 1 },       \dots,       X _{ r }       \hookrightarrow       X    \)
    for some    \(       r       \ge       2    \)    such that    \(       X       =       \left(        \bigcup        _{         i = 1        }        ^{         r        }        X        _{         i        }       \right)       \cup       E    \)    and    \(       | X _{ i } \cap E |       =       1    \)    for    \(       i       =       1,       \dots,       r    \)    (hence the name \emph{comb}).    Then it follows for each    \(       i    \)    that    \(        \dim       _{        \bfk       }       H ^{ 1 }       \left(        X _{ i },        \cO        _{            X _{ i }        }       \right)       >       0    \).    Indeed, otherwise    \(       X _{ i}    \)    is a tree of    \(       {\mathbb P} ^{ 1 }    \)s    and hence the component of    \(        X _{ i }    \)    which corresponds to a leaf of the dual graph of    \(        X _{ i }    \)    would destabilize       \(       X    \).    Hence we can apply Lemma \ref{lemma:there is a nice line bundle} to each    \(       X _{ i }    \)    to obtain a non-trivial invertible sheaf    \(       L _{ i }       \in       \Pic ^{ 0 }       \left(        X _{ i }       \right)    \).    Now let    \(       L       \in       \Pic       ^{        0       }       (        X       )    \)    be such that    \(       L       \vert       _{        X        _{            i        }       }       \simeq       L _{ i }    \)    for all    \(       i       =       1,       \dots,       r    \)    (see, say, the 1st paragraph of Example 0.7.2 in \cite{CD89} for the existence of such    \(       L    \)).

    For a general closed point    \(       x       \in       E    \)    consider the following short exact sequence
    
  \noindent
 \begin{align}        0        \to        L        \otimes        \omega        _{            X            /            \bfk        }        ( - x )        \to        L        \otimes        \omega        _{            X            /            \bfk        }        \to        \cO        _{            x        }        \to        0,
    \end{align}    from which we obtain the following exact sequence.
    
    \noindent
\begin{align}        H        ^{            0        }        \left(            X,            L            \otimes            \omega            _{                X                /                \bfk            }        \right)        \to        H        ^{            0        }        \left(            \cO            _{                x            }        \right)        \to        H        ^{            1        }        \left(            X,            L            \otimes            \omega            _{                X                /                \bfk            }            (                - x            )        \right)
    \end{align}    The last term is isomorphic to the dual of the following vector space by the Serre duality.
    
    \noindent
\begin{align}\label{equation:what we want to kill}        H        ^{            0        }        \left(            X,            L            ^{                - 1            }            (                x            )        \right)
    \end{align}    By Lemma \ref{lemma:no global section},    for each    \(       i       =       1,       \dots,       r    \)    we have    \(       H       ^{        0       }       \left(        X        _{            i        },        L        ^{            - 1        }        (            x        )        \vert        _{            X            _{                i            }        }       \right)       =       0    \).    Finally, as    \(        L        ^{            - 1        }        (            x        )        \vert        _{            E        }        =        \cO        _{            E        }        (            1        )    \)    and    \(    r    \ge    2    \),    we conclude that~\eqref{equation:what we want to kill} vanishes.

    Conversely, let    \(       X    \)    be a curve as in the assertion which is not semistable. Then there must be an irreducible component    \(       C       \hookrightarrow       X    \)    such that    \(       C       \simeq       {\mathbb P} ^{ 1 }    \)    and    \(       C    \)    meets the other irreducible components of
    \(       X    \)    in exactly    \(       1    \)    point. Then Proposition 6.15 in  \cite{KuzShi22a} implies that there is a nontrivial semiorthogonal decomposition of    \(       \cdbc{ X}    \),    which implies that    \(       \Perf (X)    \)    admits a nontrivial semiorthogonal decomposition by Corollary 6.6 in \cite{KuzShi22}.
\end{proof}

\begin{rem}    When the components    \(       X       _{        1       },       \dots,       X       _{        r       }    \)    in the proof of Proposition \ref{proposition:semistable curves are indecomposable}    are all smooth, then we can prove Proposition \ref{proposition:semistable curves are indecomposable} by means of the relative canonical base locus. This can be regarded as a sophisticated version of Lin's argument.

    Indeed, in this case we have a morphism as follows, which contracts the LCRT   \(       E    \)    in the proof of Proposition \ref{proposition:semistable curves are indecomposable}    to a point.
    
    \noindent
\begin{align}        f        \colon        X        \to        \prod        _{            i = 1        }        ^{            r        }        \Pic        ^{            0        }        \left(            X            _{                i           }        \right)
    \end{align}    One can easily show that    \(     \text{Bs} |f|    =       \emptyset    \)    (recall that    \(       r       \ge       2    \)).    Hence there is no    \(       f    \)-linear semiorthogonal decomposition of    \(       \Perf       (X)    \)    by Theorem \ref{t:relativeindescomPerf}.    Now since    \(        \prod        _{            i = 1        }        ^{            r        }        \Pic        ^{            0        }        \left(            X            _{                i            }        \right)    \)    is an abelian variety, any semiorthogonal decomposition of    \(       \Perf       (X)    \)    is    \(       f    \)-linear by Theorem 1.4 in \cite{Piroz2023}. Hence the conclusion.    It is interesting to ask if one can push the same strategy without assuming    \(       X       _{        i       }    \)    are smooth.
\end{rem}

\bibliographystyle{siam}

  \end{document}